\documentclass{birkjour}
\usepackage{graphics}
\usepackage{epsfig}
\usepackage{amsfonts}
\usepackage{amssymb}
\usepackage{amsthm}
\usepackage{newlfont}
\usepackage{slashbox}
 \newtheorem{thm}{Theorem}[section]

 \newtheorem{prop}[thm]{Proposition}
 \theoremstyle{definition}
 \newtheorem{defn}[thm]{Definition}
 \theoremstyle{remark}
 
 \newtheorem*{ex}{Example}
 \numberwithin{equation}{section}

\begin{document}

%
%
%
%
%
%
%
%
%

\title[Some Topological Invariants of Generalized M\"{o}bius Ladder]
 {\begin{center}Some Topological Invariants of Generalized M\"{o}bius Ladder \end{center}}

\author{Numan Amin}
\address{Abdus Salam School of Mathematical Sciences, GC University, Lahore-Pakistan}
\email{numan.amin@sms.edu.pk}
\author{Abdul Rauf Nizami}
\address{Abdus Salam School of Mathematical Sciences, GC University, Lahore-Pakistan}
\email{arnizami@sms.edu.pk}
\author{Muhammad Idrees}
\address{School of Automation, Beijing Institute of Technology, Beijing-China}
\email{idrees@bit.edu.cn}


\maketitle

\begin{abstract}
The Hosoya polynomial of a graph $G$ was introduced by H. Hosoya in 1988 as a counting polynomial, which actually counts the number of distances of paths of different lengths in $G$. The most interesting application of the Hosoya polynomial is that almost all distance-based graph invariants, which are used to predict physical, chemical and pharmacological properties of organic molecules, can be recovered from it. In this article we give the general closed form of the Hosoya polynomial of the generalized M\"{o}bius ladder $M(m,n)$ for arbitrary $m$ and for $n=3$. Moreover, we recover Wiener, hyper Wiener, Tratch-Stankevitch-Zefirov, and Harary indices from it.
\end{abstract}


\pagestyle{myheadings}
\markboth{\centerline {\scriptsize
Numan, Nizami, and Idrees}} {\centerline {\scriptsize
 Some Topological Invariants of Generalized M\"{o}bius Ladder}}
\section{Introduction}\label{sec1}
The Hosoya polynomial of a graph was introduced by H. Hosoya in 1988 as a counting polynomial; it actually counts the number of distances of paths of different lengths in the graph \cite{Hosoya:88}. \\

\noindent Hosoya polynomial is very well studied. In 1993, Gutman introduced Hosoya polynomial for a vertex of a graph \cite{Gutman:93}, which is related with Hosoya polynomial of the graph. The most interesting application of the Hosoya polynomial is that almost all distance-based graph invariants, which are used to predict physical, chemical and pharmacological properties of organic molecules, can be recovered from it.\\

\noindent Hosoya polynomial has been computed for several classes of graphs. In 2002 Diudea computed the Hosoya polynomial of several classes of toroidal nets and recovered their Wiener indices \cite{Diudea:02}. In 2011 Ali found the Hosoya polynomial of concatenated pentagonal rings \cite{Ali:11}. In 2012 Kishori gave a recursive method for calculating the Hosoya polynomial of Hanoi graphs, and computed some of their distance-based invariants \cite{Kishori:12}. In 2013 Farahi computed the Hosoya polynomial of polycyclic aromatic hydrocarbons \cite{Farhani:13}.\\

\noindent There are some useful topological indices that are related to Hosoya polynomial as Wiener, hyper Wiener, Tratch-Stankevitch-Ziefirov (TSZ), and Harary indices. The Wiener index was introduced by Harry Wiener in 1947 and was used to correlate with boiling points of alkanes \cite{Wiener:47}. Later it was observed that the Wiener index can be used to determine a number of physico-chemical properties of alkanes as heats of formation, heats of vaporization, molar volumes, and molar refractions \cite{Gutman:86}. Moreover, it can be used to correlate those physico-chemical properties which depend on the volume-surface ratio of molecules and to Gas-chromatographic retention data for series of structurally related molecules. Another topological index whose mathematical properties are relatively well investigated is the hyper-Wiener index and was introduced by Randic in 1993 \cite{Randic:93}. It is also used to predict physico-chemical properties of organic compounds, particularly to pharmacology, agriculture, and environment protection \cite{Cash:02}; for more details, see also \cite{Randic2:93,Lukovits:94,Klein:95,Diudea:96,Gutman:97}. In 1993 Plavsic \emph{et al.} introduced a new topological index, known as Harary index, to characterize chemical graphs \cite{Plavsic:93}. Tratch, Stankevitch, and Zefirov introduced Tratch-Stankevitch-Zefirov (TSZ) index as expanded Wiener index in 1990 \cite{Tratch:90}.\\

\noindent This article is organized as follows: Section~\ref{sec2} covers the basic definitions as of graph, distance, Hosoya polynomial, topological index, and generalized M\"{o}bius ladder; Section~\ref{sec3} contains the genral closed forms of the Hosoya polynomial of generalized M\"{o}bius ladder for arbitrary $m$ and $n=3$; Section~\ref{sec4} covers the topological indices of this ladder.

\section{Basic Definitions}\label{sec2}
\noindent A \emph{graph} $G$ is a pair $(V,E)$, where $V$ is the set of vertices and $E$ is the set of edges. A \emph{path} from a vertex $v$ to a vertex $w$ in a graph $G$ is a sequence of vertices and edges that starts from $v$ and stops at $w$. The number of edges in a path is the \emph{length} of that path. A graph is said to be \emph{connected} if there is a path between any two of its vertices. The \emph{distance} $d(u,v)$ between two vertices $u,v$ of a connected graph $G$ is the length of a shortest path between them. The \emph{diameter} of $G$, denoted by $d(G)$, is the longest distance in $G$. 
\begin{center}
\begin{minipage}{5cm}
\centering  \epsfig{figure=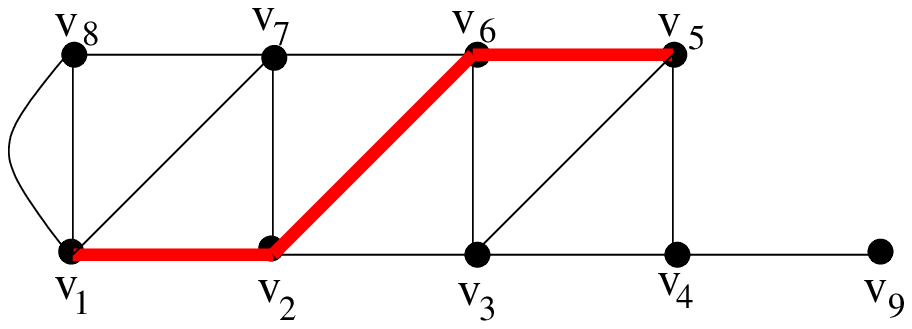,height=1.5cm}
\par \tiny{Figure 1: A connected graph with a highlighted shortest path from $v_{1}$ to $v_{5}$}
\end{minipage}
\end{center}
\noindent A \emph{molecular graph} is a representation of a chemical compound in terms of graph theory. Specifically, molecular graph is a graph whose vertices correspond to (carbon) atoms of the compound and whose edges correspond to chemical bonds. For instance, Figure 2 represents the molecular graph of 1-bromopropyne ($CH_{3}-C\equiv C-Br$).
\begin{center}
\begin{minipage}{10cm}
\centering  \epsfig{figure=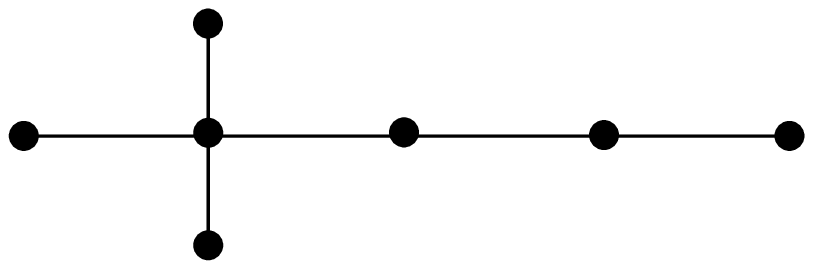,height=1.5cm}
\par \tiny{Figure 2: The molecular graph of 1-bromopropyne}
\end{minipage}
\end{center}
\begin{defn}
The \emph{Hosoya polynomial} in a variable $x$ of a molecular graph $G=(V,E)$ is defined as $$H(G,x) = \sum_{\{v,u\}\in V} x^{d(u,v)}=\sum_{k=1}^{d(G)} d(G,k)x^{k},$$
where $d(G,k)$ is the number of pairs of vertices of $G$ laying at distance $k$ from each other.
\end{defn}

\noindent A function $I$ which assigns to every connected graph $G$ a unique number $I(G)$ is called a \emph{graph invariant}. Instead of the function $I$ it is custom to say the number $I(G)$ as the invariant. An invariant of a molecular graph which can be used to determine structure-property or structure-activity correlation is called the \emph{topological index}. A topological index is said to be\emph{distance-based} if it depends on paths in the graph.\\

\noindent The following are definitions of some those distance-based indices that have connections with the Hosoya polynomial.\\

\noindent Let $u,v$ be arbitrary vertices of a connected graph $G=(V,E)$, and let $d(v,G)$ is the sum of distances of $v$ with all vertices of $G$. The \emph{Wiener index} $W(G)$ of the graph $G$ is defined as $$W(G) = \sum_{v<u;u,v\in V} d(v,u)= \frac{1}{2}\sum_{v\in V} d(v,G).$$
\noindent The Wiener index and the Hosoya polynomial are related by the equation
 $$W(G) = \frac{d}{dx}H(G,x)|_{x=1}.$$

\noindent The \emph{hyper-Wiener index} $WW(G)$ of a graph $G$ is defined as $$WW(G) = \sum_{v<u;u,v\in V} d(v,u)= \frac{1}{2}\sum d(v,u)^{2}+\frac{1}{2}\sum d(v,u).$$
\noindent The hyper-Wiener index and the Hosoya polynomial are related by the equation
 $$WW(G) = \frac{1}{2}\frac{d^{2}}{dx^{2}}xH(G,x)|_{x=1}.$$
 
\noindent The \emph{Harary index} $Ha(G)$ of a graph $G$ is defined as $$Ha(G) = \sum_{i<j}\frac{1}{d(u_{i},v_{j})^{2}},$$
and is related to Hosoya polynomial by 
 $$Ha(G) = \int_{0}^{1}\frac{H(G,x)}{x}dx.$$

\noindent The Tratch-Stankevitch-Zefirov index is also related to the Hosoya polynomial under the relation
 $$TSZ(G) = \frac{1}{3!}\frac{d^{3}}{dx^{3}}x^{2}H(G,x)|_{x=1}.$$

 \begin{defn}\cite{Hongbin-Idrees-Nizami-Munir:17} Consider the Cartesian product $P_{m}\times P_{n}$ of paths $P_{m}$ and $P_{n}$ with vertices $u_{1},u_{2},\ldots,u_{m}$ and $v_{1},v_{2},\ldots,u_{n}$, respectively. Take a $180^{o}$ twist and identify the vertices $(u_{1},v_{1}),(u_{1},v_{2}),\ldots,(u_{1},v_{n})$ with the vertices $(u_{m},v_{n})$, $(u_{m},v_{n-1})$, $\ldots,(u_{m},v_{1})$, respectively, and identify the edge $\big((u_{1},i)$, $(u_{1},i+1)\big)$ with the edge $\big((u_{m},v_{n+1-i})$, $(u_{m},v_{n-i})\big)$, where $1\leq i\leq n-1$. What we receive is the \emph{generalized M\"{o}bius ladder} $M_{m,n}$.
\end{defn}

\noindent You can see $M_{7,3}$ in the following figure.
 \begin{center}
\begin{minipage}{10cm}
\centering  \epsfig{figure=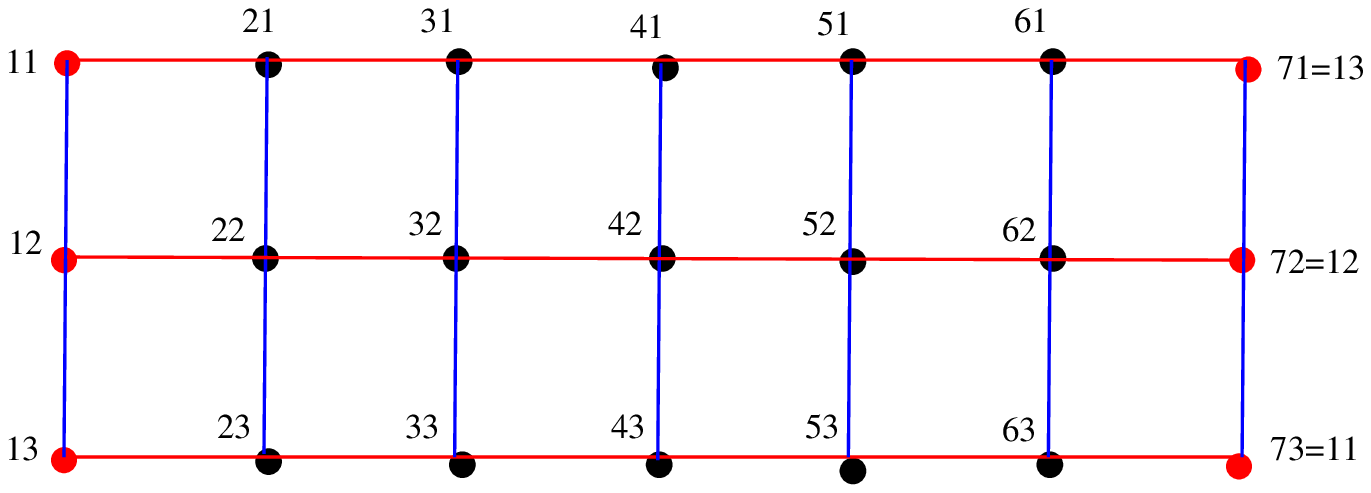, height=2.5cm}
\par \tiny {Figure 3: The grid form of the generalized M\"{o}bius ladder $M_{7,3}$ }
\end{minipage}
\end{center}

The original form of $M_{7,3}$ is:
\begin{center}
\begin{minipage}{10cm}
\centering  \epsfig{figure=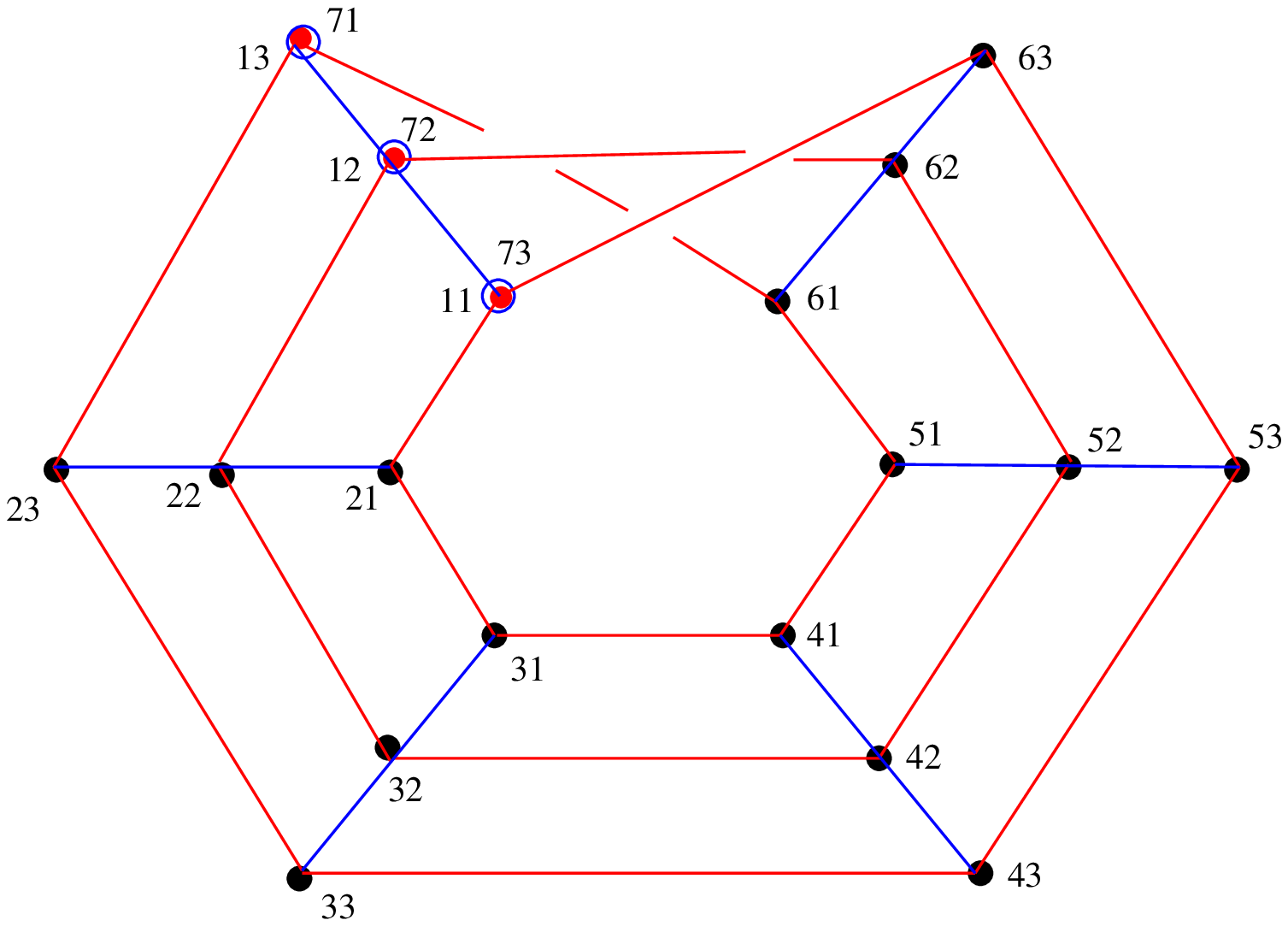, height=3.0cm}
\par Figure 4: The generalized M\"{o}bius ladder $M_{7,3}$
\end{minipage}
\end{center}

\section{Main Results}\label{sec3}
In this section we give the closed form of the Hosoya polynomial of the generalized M\"{o}bius ladder $M_{m,n}$ for $n=3$. Actually, we give two results, one for even $m$ and other for odd $m$.
\begin{thm}\label{thm3.1}
 The Hosoya polynomial of the $M_{m,3}$, when $m\geq 6$ is even, is $H(M_{m,3})=\sum_{k=1}^{\frac{m}{2}}c_{k}x^{k}$, where
  $$ c_{k}=\left\{
                  \begin{array}{ll}
                    5(m-1), &  k=1  \\
                    8(m-1), & k=2\\
                    9(m-1), &  2<k< \frac{m}{2}\\
                    8(m-1), & k=\frac{m}{2}.
                  \end{array}
                \right.$$
 \end{thm}
\begin{proof}
In order to find the $k$th coefficient of the Hosoya polynomial we first give the upper triangular entries of the $3(m-1) \times 3(m-1)$ distance matrix $D=(d_{ij})$ corresponding to $M_{m,3}$ and then find the sum of all the entries $d_{ij}=k,j>i,$ to find $c_{k}$.\\

\noindent Since $D$ is symmetric, we only need to know its upper-triangular part. For this we use only $m-1$ matrices $B_{l}, 0\leq q \leq m-2$, each having order $3\times 3$. Each $B_{q}$ appears $m-q-1$ times on the $q$th secondary diagonal; by the $0$th secondary diagonal we mean the main diagonal. These matrices are\\

\noindent $B_{0}=\left(
  \begin{array}{ccc}
    0 & 1 & 2 \\
    . & 0 & 1 \\
    . & . & 0 \\
  \end{array}
\right),$ 
$B_{q}=\left(
  \begin{array}{ccc}
    q-1 & q & q+1 \\
    q & q-1 & q \\
    q+1 & q & q-1 \\
  \end{array}
\right)$ for $q=1,2,\dots,\frac{m-4}{2},$
$B_{\frac{m-2}{2}}=\left(
  \begin{array}{ccc}
    \frac{m-4}{2} & \frac{m-2}{2} & \frac{m-2}{2} \\
    \frac{m-2}{2} & \frac{m-4}{2} & \frac{m-2}{2} \\
    \frac{m-2}{2} & \frac{m-2}{2} & \frac{m-4}{2} \\
  \end{array}
\right),$
$B_{\frac{m}{2}}=\left(
  \begin{array}{ccc}
    \frac{m+2}{2} & \frac{m+2}{2} & \frac{m}{2} \\
    \frac{m+2}{2} & \frac{m}{2} & \frac{m+2}{2} \\
    \frac{m}{2} & \frac{m+2}{2} & \frac{m+2}{2} \\
  \end{array}
\right),$
and
$B_{q}=\left(
  \begin{array}{ccc}
    m-q+2 & m-q+1 & m-q \\
    m-q+1 & m-q & m-q+1 \\
    m-q & m-q+1 & m-q+2 \\
  \end{array}
\right)$ for $q=\frac{m+2}{2},\frac{m+4}{2},\ldots,m-2.$\\

\noindent We now give the coefficients $c_{k}$ of the polynomial. Depending on the special behavior, we give $c_{1}$, $c_{2}$, and $c_{\frac{m}{2}}$ one-by-one, and give all the remaining in a single general form:\\

\noindent $c_{1}$: The entry $1$ appears only in $B_{0}$, $B_{1}$, and $B_{m-2}$. Since $1$ appears in each $B_{0}$ twice and there are $m-1$ $B_{0}$ in $D$, the number of $1$s in all $B_{0}s$ is $2(m-1)$. Since $1$ appears in each $B_{1}$ thrice and there are $m-2$ $B_{1}s$ in $D$, the total number of $1$s in $B_{1}s$ is $3(m-2)$. Since $1$ appears in each $B_{m-2}$ thrice and there is only one $B_{m-2}$ in $D$, the number of $1$s in $B_{m-2}$ is $3$. Hence $c_{1}=2(m-1)+3(m-2)+3=5(m-1)$.\\

\noindent $c_{2}$: The entry $2$ appears in $B_{0}$, $B_{1}$, $B_{2}$, $B_{m-3}$, and $B_{m-2}$. Since $2$ appears in each $B_{0}$ once and there are $m-1$ $B_{0}s$ in $D$, the number of $2$s in all $B_{0}s$ is $m-1$. Since $2$ appears in each $B_{1}$ four times and there are $m-2$ $B_{1}s$ in $D$, the number of $2$s in all $B_{1}s$ is $4(m-2)$. Since $2$ appears in each $B_{2}$ thrice and there are $m-3$ $B_{2}s$ in $D$, the number of $2$s in all $B_{2}s$ is $3(m-3)$. Since $2$ appears in each $B_{m-3}$ thrice and there are $2$ $B_{m-3}s$ in $D$, the number of $2$s in all $B_{m-3}s$ is $6$. Since $2$ appears in each $B_{m-2}$ four times and there is $1$ $B_{m-2}s$ in $D$, the number of $2$s in  $B_{m-2}s$ is $4$. Hence, $c_{2}=(m-1)+4(m-2)+3(m-3)+2(3)+4=8(m-1)$.\\

\noindent $c_{\frac{m}{2}}$: The entry $\frac{m}{2}$ appears $2$ times in $B_{\frac{m-4}{2}}$,  $2$ times in $B_{\frac{m+2}{2}}$, $6$ times in $B_{\frac{m-2}{2}}$,  and $6$ times in $B_{\frac{m}{2}}$. Since the matrices $B_{\frac{m-4}{2}}$, $B_{\frac{m+2}{2}}$, $B_{\frac{m-2}{2}}$,  and $B_{\frac{m}{2}}$ appear respectively   $\frac{m+2}{2}$, $\frac{m-4}{2}$, $\frac{m}{2}$, and  $\frac{m-2}{2}$ times in $D$, we have $c_{\frac{m}{2}}=2(\frac{m+2}{2})+2(\frac{m-4}{2})+6(\frac{m}{2})+6(\frac{m-2}{2})=8(m-1)$.\\

\noindent $c_{k}, 3\leq k\leq \frac{m-2}{2}$: The entry $k$ appears $2$ times in $B_{k-1}$, $2$ times in $B_{m-k+2}$, $4$ times in $B_{k}$, $4$ times in $B_{m-k+1}$, $3$ times in $B_{k+1}$,  and $3$ times in $B_{m-k}$. The matrices $B_{k-1}$, $B_{m-k+2}$, $B_{k}$,  $B_{m-k+1}$, $B_{k+1}$, and $B_{m-k}$ appear respectively $m-k+1$, $k-2$, $m-k$, $k-1$, $m-k-1$, and $k$ times in $D$. Hence, $c_{k}=2(m-k+1)+2(k-2)+4(m-k)+4(k-1)+3(m-k-1)+3k=9(m-1)$, and we are done.
\end{proof}
For better understanding, let us have a look at the example:
\begin{ex}
Consider $M_{10,3}$: 
\begin{center}
\begin{minipage}{12cm}
\centering  \epsfig{figure=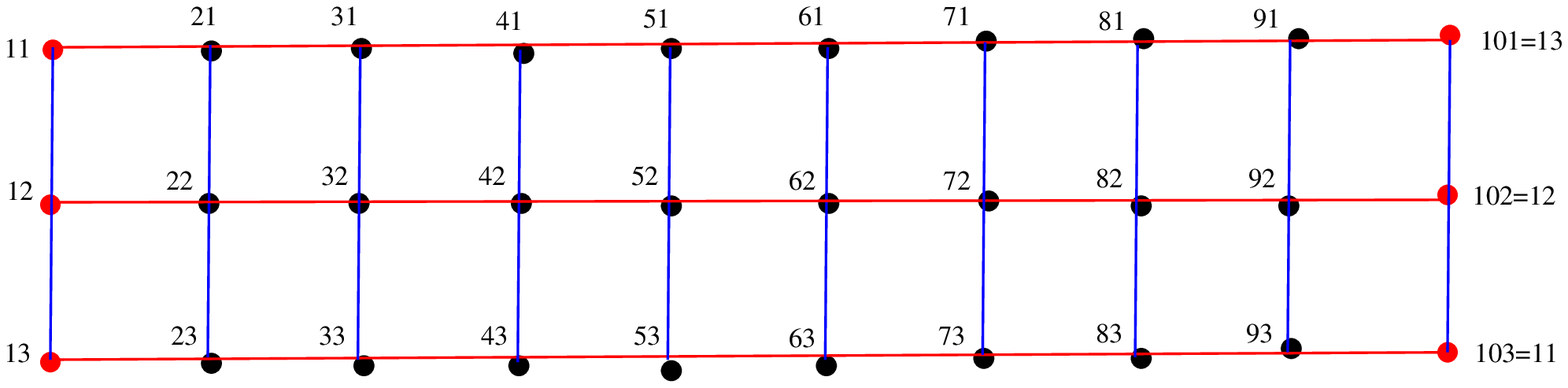, height=2.0cm}
\par \tiny {Figure 5: The grid form of $M_{10,3}$ }
\end{minipage}
\end{center}
\begin{minipage}[]{4cm}
The block matrices are: $$B_{0}=\left(
  \begin{array}{ccc}
    0 & 1 & 2 \\
    . & 0 & 1 \\
    . & . & 0 \\
  \end{array}
\right)$$
\end{minipage}
\begin{minipage}[]{4cm}
$$B_{1}=\left(
  \begin{array}{ccc}
    1 & 2 & 3 \\
    2 & 1 & 2 \\
    3 & 2 & 1 \\
  \end{array}
\right)$$
\end{minipage}
\begin{minipage}[]{4cm}
$$B_{2}=\left(
  \begin{array}{ccc}
    2 & 3 & 4 \\
    3 & 2 & 3 \\
    4 & 3 & 2 \\
  \end{array}
\right)$$
\end{minipage}
\begin{minipage}[]{4cm}
$$B_{3}=\left(
  \begin{array}{ccc}
    3 & 4 & 5 \\
    4 & 3 & 4 \\
    5 & 4 & 3 \\
  \end{array}
\right)$$
\end{minipage}
\begin{minipage}[]{4cm}
$$B_{4}=\left(
  \begin{array}{ccc}
    4 & 5 & 5 \\
    5 & 4 & 5 \\
    5 & 5 & 4 \\
  \end{array}
\right)$$
\end{minipage}
\begin{minipage}[]{4cm}
$$B_{5}=\left(
  \begin{array}{ccc}
    5 & 5 & 4 \\
    5 & 4 & 5 \\
    4 & 5 & 5 \\
  \end{array}
\right)$$
\end{minipage}
\begin{minipage}[]{4cm}
$$B_{6}=\left(
  \begin{array}{ccc}
    5 & 4 & 3 \\
    4 & 3 & 4 \\
    3 & 4 & 5 \\
  \end{array}
\right)$$
\end{minipage}
\begin{minipage}[]{4cm}
$$B_{7}=\left(
  \begin{array}{ccc}
    4 & 3 & 2 \\
    3 & 2 & 3 \\
    2 & 3 & 4 \\
  \end{array}
\right)$$
\end{minipage}
\begin{minipage}[]{4cm}
$$B_{8}=\left(
  \begin{array}{ccc}
    3 & 2 & 1 \\
    2 & 1 & 2 \\
    1 & 2 & 3 \\
  \end{array}
\right)$$
\end{minipage}\\

\noindent The distance matrix $D$ in the form of these block matrices is 

$$D=\left(
  \begin{array}{ccccccccc}
    B_{0} & B_{1} & B_{2} & B_{3} & B_{4} & B_{5} & B_{6} & B_{7} & B_{8} \\
    . & B_{0} & B_{1} & B_{2} & B_{3} & B_{4} & B_{5} & B_{6} & B_{7}\\
    . & . & B_{0} & B_{1} & B_{2} & B_{3} & B_{4} & B_{5} & B_{6} \\
    . & . & . & B_{0} & B_{1} & B_{2} & B_{3} & B_{4} & B_{5}\\
    . & . & . & . & B_{0} & B_{1} & B_{2} & B_{3} & B_{4}\\
    . & . & . & . & . & B_{0} & B_{1} & B_{2} & B_{3}\\
    . & . & . & . & . & . & B_{0} & B_{1} & B_{2}\\
    . & . & . & . & . & . & . & B_{0} & B_{1}\\
    . & . & . & . & . & . & . & . & B_{0} \\
  \end{array}
\right).$$

\noindent It is now evident from $D$ that
\begin{eqnarray*}
  H(M_{10,3}) &=& c_{1}x^{1}+c_{2}x^{2}+c_{3}x^{3}+c_{4}x^{4}+c_{5}x^{5} \\
   &=& 5(m-1)x^{1}+8(m-1)x^{2}+9(m-1)x^{3}+9(m-1)x^{4}+8(m-1)x^{5}\\
   &=& 45x^{1}+72x^{2}+81x^{3}+81x^{4}+72x^{5}.
\end{eqnarray*}
\end{ex}
For $m=4$ we receive the special case:
\begin{prop}
$H(M_{4,3})=15x+21x^2$.
\end{prop}
\begin{proof}
\noindent Since the distance matrix $D$ corresponding to $M_{4,3}$ is symmetric, we give only its upper-triangular part:
$$D=\left(
      \begin{array}{ccccccccc}
        0 & 1 & 2 & 1 & 2 & 2 & 2 & 2 & 1 \\
         & 0 & 1 & 2 & 1 & 2 & 2 & 1 & 2 \\
         &  & 0 & 2 & 2 & 1 & 1 & 2 & 2 \\
         &  &  & 0 & 1 & 2 & 1 & 2 & 2 \\
         &  &  &  & 0 & 1 & 2 & 1 & 2 \\
         &  &  &  &  & 0 & 2 & 2 & 1 \\
         &  &  &  &  &  & 0 & 1 & 2 \\
         &  &  &  &  &  &  & 0 & 1 \\
         &  &  &  &  &  &  &  & 0 \\
      \end{array}
    \right)$$
\noindent It immediately follows that $c_{1}=15$ and $c_{2}=21$.
\end{proof}
For odd $m$ we have the result:
\begin{thm}\label{thm3.2}
 The Hosoya polynomial of $M_{m,3}$, when $m\geq 7$ is odd, is $H(M_{m,3})=\sum_{k=1}^{\frac{m+1}{2}}c_{k}x^{k}$, where
  $$ c_{k}=\left\{
                  \begin{array}{ll}
                    5(m-1), &  k=1  \\
                    8(m-1), & k=2\\
                    9(m-1), &  2<k< \frac{m-1}{2}\\
                    17(\frac{m-1}{2}),& k=\frac{m-1}{2}\\
                    4(m-1), & k=\frac{m+1}{2}.
                  \end{array}
                \right.$$
 \end{thm}
\begin{proof}
The block matrices, which were explained in Theorem~\ref{thm3.1}, of the distance matric of $M_{m,3}$, when $m\geq 7$ is odd, are:   \noindent $B_{0}=\left(
  \begin{array}{ccc}
    0 & 1 & 2 \\
    . & 0 & 1 \\
    . & . & 0 \\
  \end{array}
\right),$
$B_{q}=\left(
  \begin{array}{ccc}
    q-1 & q & q+1 \\
    q & q-1 & q \\
    q+1 & q & q-1 \\
  \end{array}
\right)$ for $q=1,2,\dots,\frac{m-3}{2},$
$B_{\frac{m-1}{2}}=\left(
  \begin{array}{ccc}
    \frac{m-3}{2} & \frac{m-1}{2} & \frac{m-3}{2} \\
    \frac{m-1}{2} & \frac{m-3}{2} & \frac{m-1}{2} \\
    \frac{m-3}{2} & \frac{m-1}{2} & \frac{m-3}{2} \\
  \end{array}
\right),$
and
$B_{q}=\left(
  \begin{array}{ccc}
    m-q+2 & m-q+1 & m-q \\
    m-q+1 & m-q & m-q+1 \\
    m-q & m-q+1 & m-q+2 \\
  \end{array}
\right)$ for $q=\frac{m+1}{2},\frac{m+3}{2},\ldots,m-1.$\\

The proofs $c_{1}$, $c_{2}$, and $c_{k}, 2<k< \frac{m-1}{2}$, are similar to the proofs given in Theorem~\ref{thm3.1}. 

\noindent We need to find $c_{\frac{m-1}{2}}$ and $c_{\frac{m+1}{2}}$. The entry $\frac{m-1}{2}$ appears $5$ times in $B_{\frac{m-1}{2}}$, $4$ times in $B_{\frac{m+1}{2}}$, $4$ times in $B_{\frac{m-3}{2}}$, $2$ times in $B_{\frac{m-5}{2}}$, and $2$ times in $B_{\frac{m+3}{2}}$. Since $B_{\frac{m-1}{2}}$, $B_{\frac{m+1}{2}}$,  $B_{\frac{m-3}{2}}$, $B_{\frac{m-5}{2}}$, and $B_{\frac{m+3}{2}}$ appear respectively $\frac{m-1}{2}$, $\frac{m-3}{2}$, $\frac{m+1}{2}$, $\frac{m+3}{2}$, and $\frac{m-5}{2}$. Hence $c_{\frac{m-1}{2}}=5(\frac{m-1}{2})+4(\frac{m-3}{2})+4(\frac{m+1}{2})+2(\frac{m+3}{2})+2(\frac{m-5}{2})=\frac{17}{2}(m-1)$.

\noindent Finally, we go for $c_{\frac{m+1}{2}}$. The entry $\frac{m+1}{2}$ appears $4$ times in $B_{\frac{m-1}{2}}$,  $2$ times in $B_{\frac{m+1}{2}}$, and $2$ times in $B_{\frac{m-3}{2}}$. Since $B_{\frac{m-1}{2}}$,  $B_{\frac{m+1}{2}}$, and $B_{\frac{m-3}{2}}$ repeat respectively $\frac{m-1}{2}$, $\frac{m-3}{2}$, and $\frac{m+1}{2}$ times in $D$. Therefore, $c_{\frac{m+1}{2}}=4(\frac{m-1}{2})+2(\frac{m-3}{2})+2(\frac{m+1}{2})=4(m-1)$. This completes the proof.
\end{proof}
Here, again, we receive a special case:
\begin{prop}
$H(M_{5,3},x)=20x+30x^2+16x^3$.
\end{prop}
\begin{proof}
\noindent Here $D$ is:
$$D=\left(
      \begin{array}{cccccccccccc}
        0 & 1 & 2 & 1 & 2 & 3 & 2 & 3 & 2 & 3 & 2 & 1 \\
         & 0 & 1 & 2 & 1 & 2 & 3 & 2 & 3 & 2 & 1 & 2 \\
         &  & 0 & 3 & 2 & 1 & 2 & 3 & 2 & 1 & 2 & 2 \\
         &  &  & 0 & 1 & 2 & 1 & 2 & 3 & 2 & 3 & 2\\
         &  &  &  & 0 & 1 & 2 & 1 & 2 & 3 & 2 & 3 \\
         &  &  &  &  & 0 & 3 & 2 & 1 & 2 & 3 & 2 \\
         &  &  &  &  &  & 0 & 1 & 2 & 1 & 2 & 3 \\
         &  &  &  &  &  &  & 0 & 1 & 2 & 1 & 2 \\
         &  &  &  &  &  &  &  & 0 & 3 & 2 & 1 \\
         &  &  &  &  &  &  &  &  & 0 & 1 & 2 \\
         &  &  &  &  &  &  &  &  &  & 0 & 1 \\
         &  &  &  &  &  &  &  &  &  &  & 0 \\
      \end{array}
    \right)$$
\noindent It is now evident that $c_{1}=20$, $c_{2}=30$, and $c_{3}=16$.
\end{proof}
\section{Topological Indices}\label{sec4}
In this section we give the distance-based topological indices, Wiener, hyper Wiener, Harary, and  Tratch-Stankevitch-Zefirovof, of $M_{m,3}$ for both even and odd $m$.
\begin{prop}\label{prop4.1}
For even $m$ we get
\begin{enumerate}
  \item $W(M_{m,3})=\frac{1}{8}[9m^3+5m^2-62m+48]$
  \item $WW(M_{m,3})=\frac{3}{16}m^4+\frac{13}{16}m^3+\frac{1}{4}m^2-\frac{33}{4}m+7$
  \item $H_{a}(M_{m,3})=9m+7-\frac{16}{m}++9(m-1)\sum_{i=3}^{\frac{m-2}{2}}\frac{1}{i}$
  \item $TSZ(M_{m,3})=\frac{1}{16}m^4+\frac{43}{24}m^2+\frac{31}{48}m^3+\frac{19}{3}-\frac{53}{6}m$
\end{enumerate}
\end{prop}
\begin{proof} We prove these relations one by one using Theorem~\ref{thm3.1} and the relations given in Introduction section:
\begin{enumerate}
 \item \begin{eqnarray*}
    W(M_{m,3}) &=& \frac{d}{dx}[H(M_{m,3},x)]_{x=1} \\
     &=& \frac{d}{dx}[5(m-1)x+8(m-1)x^2+8(m-1)x^{\frac{m}{2}}\\
     &&  +\sum_{i=3}^{\frac{m-2}{2}}9(m-1)x^{i}]_{x=1} \\
     &=& 21m-21+4m(m-1)+9(m-1)\sum_{i=3}^{\frac{m-2}{2}}i\\
     &=& \frac{1}{8}[9m^3+5m^2-62m+48]
  \end{eqnarray*}

\item \begin{eqnarray*}
    WW(M_{m,3})&=&\frac{1}{2}\frac{d^2}{dx^2}[xH(M_{m,3},x)]_{x=1}\\
     &=& \frac{1}{2}\frac{d^2}{dx^2}[5(m-1)x^2+8(m-1)x^3+8(m-1)x^{(\frac{m}{2}}+1)\\
     &&+\sum_{i=3}^{\frac{m-2}{2}}9(m-1)x^{i+1}]_{x=1}\\
     &=& \frac{1}{2}[10(m-1)+48(m-1)+4m(m-1)(m+2)\\
     &&+9(m-1)\sum_{i=3}^{\frac{m-2}{2}}i{i+1}] \\
     &=& \frac{3}{16}m^4+\frac{13}{16}m^3+\frac{1}{4}m^2-\frac{33}{4}m+7
   \end{eqnarray*}

   \item \begin{eqnarray*}
    H_{a}(M_{m,3})&=&\int_{0}^{1}\frac{H(M_{m,3},x)}{x}dx\\
     &=& \int_{0}^{1}\frac{5(m-1)x+8(m-1)x^2+8(m-1)x^{\frac{m}{2}}+\sum_{i=3}^{\frac{m-2}{2}}9(m-1)x^{i}}{x}dx\\
     &=& 9m+7-\frac{16}{m}+9(m-1)\sum_{i=3}^{\frac{m-2}{2}}\frac{1}{i}
   \end{eqnarray*}

\item \begin{eqnarray*}
TSZ(M_{m,3})&=&\frac{1}{3!}\frac{d^3}{dx^3}[x^2H(M_{m,3})]_{x=1}\\
     &=& \frac{1}{3!}\frac{d^3}{dx^3}[5(m-1)x^3+8(m-1)x^4+\sum_{i=3}^{\frac{m-2}{2}}9(m-1)x^{i+2}\\
     && +8(m-1)x^{\frac{m}{2}+2}]_{x=1}\\
     &=& \frac{1}{3!}[30(m-1)+96(m-1)\\
     && +\sum_{i=3}^{\frac{m-2}{2}}i(i+1)(i+2)9(m-1)]+2(m+2)(m+4)(m-1)\\
     &=& \frac{1}{16}m^4+\frac{43}{24}m^2+\frac{31}{48}m^3+\frac{19}{3}-\frac{53}{6}m
\end{eqnarray*}
\end{enumerate}
\end{proof}

\begin{prop}\label{prop4.2}
For odd $m$ we get
\begin{enumerate}
  \item $W(M_{m,3})=\frac{1}{8}[9m^3+5m^2-53m+39]$
  \item $WW(M_{m,3})=\frac{3}{16}m^4+\frac{13}{16}m^3+\frac{13}{16}m^2-\frac{125}{16}m+6$
  \item $H_{a}(M_{m,3})=\frac{m(9m+25)}{m+1}+9(m-1))\sum_{i=3}^{\frac{m-2}{2}}\frac{1}{i}$
  \item $TSZ(M_{m,3})=\frac{1}{16}m^4+\frac{31}{48}m^3+\frac{95}{48}m^2-\frac{133}{16}m+\frac{45}{8}$
\end{enumerate}
\end{prop}

\begin{proof} The proof is similar to the proof of Proposition~\ref{prop4.1}.
\end{proof}

\end{document}